\documentclass[11pt,a4paper]{article}

\usepackage{inputenc}
\usepackage{amsmath}
\usepackage{bm}
\usepackage{bbold}
\usepackage{amsthm}
\usepackage{enumerate}

\usepackage[hyphens]{url}

\usepackage{hyperref}
\usepackage{breakurl}

\title{Using parameter elimination to solve discrete linear Chebyshev approximation problems
\thanks{Mathematics, 2020, 8(12), 2210, doi:10.3390/math8122210}}

\author{N. Krivulin\thanks{Faculty of Mathematics and Mechanics, Saint Petersburg State University, 28 Universitetsky Ave., St.~Petersburg, 198504, Russia, 
nkk@math.spbu.ru.}
\thanks{This work was supported in part by the Russian Foundation for Basic Research (grant No. 20-010-00145).}}

\date{}

\newtheorem{theorem}{Theorem}
\newtheorem{lemma}[theorem]{Lemma}

\newtheorem{proposition}[theorem]{Proposition}

\theoremstyle{definition}
\newtheorem{example}{Example}

\begin{document}

\maketitle

\begin{abstract}
We consider discrete linear Chebyshev approximation problems in which the unknown parameters of linear function are fitted by minimizing the maximum absolute deviation of errors. Such problems find application in the solution of overdetermined systems of linear equations that appear in many practical contexts. The least maximum absolute deviation estimator is used in regression analysis in statistics when the distribution of errors has bounded support. To derive a direct solution of the problem, we propose an algebraic approach based on a parameter elimination technique. As a key component of the approach, an elimination lemma is proved to handle the problem by reducing it to a problem with one parameter eliminated, together with a box constraint imposed on this parameter. We demonstrate the application of the lemma to the direct solution of linear regression problems with one and two parameters. We develop a procedure to solve multidimensional approximation (multiple linear regression) problems in a finite number of steps. The procedure follows a method that comprises two phases: backward elimination and forward substitution of parameters. We describe the main components of the procedure and estimate its computational complexity. We implement symbolic computations in MATLAB to obtain exact solutions for two numerical examples.
\\

\\

\end{abstract}

\section{Introduction}

Discrete linear Chebyshev (minimax) approximation problems where the errors of fitting the unknown parameters are measured by the Chebyshev (max, infinity, uniform or $L_{\infty}$) norm are of theoretical interest and practical importance in many areas of science and engineering. Application of the Chebyshev norm leads to the least maximum absolute deviation of errors as the approximation criterion, and dates back to Laplace's classical work \cite{Laplace1832Mecanique} (book~3, chap.~V, \S39) (see also \cite{Harter1975MethodI,Steffens2006History}). 

An important area of applications of the discrete linear Chebyshev approximation is the solution of overdetermined systems of linear equations \cite{Tewarson1972Minimax,Appa1973L1,Pinar2009Overdetermined} that appear in many practical contexts. The least maximum absolute deviation estimator is widely used in regression analysis in statistics when the distribution of errors has bounded support. Specifically, the Chebyshev estimator is known to be a maximum likelihood estimator if the error distribution is uniform \cite{Rabinowitz1968Applications,Appa1973L1,Hand1978Aspects,Kennedy1980Statistical}. Moreover, this estimator can be useful even if errors are not uniform, but controlled in some way, and small, relative to the observed values. Examples of applications include problems in nuclear physics \cite{James1983Fitting,Bertsch2005Fitting}, parameter estimation of dynamic systems \cite{Makila1991Robust,Akcay1996Choice}, statistical machine learning \cite{Scholkopf2002Learning,Bartoszuk2016Fitting}, and finance \cite{Jaschke1997Arbitrage}.

To solve the Chebyshev approximation problem, a number of approaches are known which apply various iterative computational procedures to find numerical solutions (see a comprehensive overview of the algorithmic solutions given by \cite{Harter1975MethodIII,Harter1975MethodIV,Hand1978Aspects,Spath1992Mathematical}). For instance, the approximation problems under consideration can be reduced to linear programs and then solved numerically by computational algorithms available in the linear programming, such as the simplex algorithm and its variations. For linear programming solutions and other related algorithms, one can consult early works \cite{Wagner1959Linear,Stiefel1960Note,Rabinowitz1968Applications,Tewarson1972Minimax,Appa1973L1,Watson1973Best,Sposito1976Minimizing,Armstrong1979Algorithm}, as well as more recent publications \cite{Kim2000Algorithms,Boyd2004Convex,Castillo2008Dealing,Ene2019Improved}.

Along with existing iterative algorithms that find use in applications, direct analytical solutions of the linear Chebyshev approximation problem are also of interest as an essential instrument of formal analysis and treatment of the problem. A useful algebraic approach to derive direct solutions of problems that involve minimizing the Chebyshev distance is proposed in \cite{Krivulin2014Complete,Krivulin2018Algebraic,Krivulin2020Algebraic}. The approach offers complete solutions of the problems in the framework of tropical (idempotent) algebra, which deals with algebraic systems with idempotent operations. The solutions obtained in terms of tropical algebra are then represented in the usual form, ready for computation.

In this paper, we reshape and adjust algebraic techniques implemented in the above-mentioned approach to develop a direct solution of the discrete linear Chebyshev approximation problem in terms of conventional algebra. As a key component of the proposed method, an elimination lemma is proved that allows us to handle the problem by reducing it to a problem with one unknown parameter eliminated and a box constraint imposed on this parameter. To provide illuminating but not cumbersome examples of the application of the lemma, we derive direct solutions of problems of low dimension, formulated as linear regression problems with one and two parameters.

Furthermore, we construct a procedure to solve multidimensional approximation (multiple linear regression) problems. The procedure is based on a direct solution method that comprises two phases: backward elimination and forward determination (substitution) of the unknown parameters. The~direct solution can supplement and complement the existing iterative procedures and becomes of particular interest when, for one reason or another, the use of iterative algorithms appears to be inappropriate or inadequate. We estimate computational complexity and memory requirements of the procedure, and implement symbolic computations in the MATLAB environment to obtain exact solutions for two illustrative numerical examples.

The rest of the paper is organized as follows. In Section~\ref{S-LCAP}, we formulate the approximation problem of interest in both scalar and vector form. Section~\ref{S-EL} presents the main result, which offers a reduction step to separate the problem into a problem of lower dimension by eliminating an unknown parameter, and a box constraint for this parameter. We apply the obtained result in Section~\ref{S-SOTPRP} to derive direct explicit solutions for linear regression problems with one and two unknown parameters. In Section~\ref{S-GSP}, we describe a computational procedure to solve linear approximation problems of arbitrary dimension and discuss its computational complexity. In Section~\ref{S-SINE}, software implementation is discussed and numerical examples are given. Section~\ref{S-C} presents concluding remarks.

\section{Linear Chebyshev Approximation Problem}
\label{S-LCAP}

We start with an appropriate notation, preliminary assumptions, and formal representation of the discrete linear Chebyshev approximation problem under study. 

Suppose that, given $X_{ij},Y_{i}\in\mathbb{R}$ for all $i=1,\ldots,M$ and $j=1,\ldots,N$, where $M$ and $N$ are positive integers, we need to find the unknown parameters $\theta_{j}\in\mathbb{R}$ for all $j=1,\ldots,N$ that solve the minimax~problem
\begin{equation}
\begin{aligned}
&&
\min_{\theta_{1},\ldots,\theta_{N}}
&&&
\max_{1\leq i\leq M}
\left|
\sum_{j=1}^{N}
X_{ij}\theta_{j}-Y_{i}
\right|.
\end{aligned}
\label{P-mintheta1thetaN}
\end{equation}

Without loss of generality, we assume that for each $j=1,\ldots,N$, there exists at least one $i$, such that $X_{ij}\ne0$. Otherwise, if $X_{ij}=0$ for some $j$ and all $i$, the parameter $\theta_{j}$ does not affect the objective function, and thus can be removed. 

Note that we can represent problem \eqref{P-mintheta1thetaN} in vector form by introducing the matrix and column~vectors
\begin{equation*}
\bm{X}
=
(X_{ij}),
\qquad
\bm{Y}
=
(Y_{i}),
\qquad
\bm{\theta}
=
(\theta_{j}).
\end{equation*}

With this matrix-vector notation and the Chebyshev norm defined for any vector $\bm{V}=(V_{i})$ as
\begin{equation*}
\|\bm{V}\|
=
\max_{i}|V_{i}|,
\end{equation*} 
the approximation problem takes the form
\begin{equation*}
\begin{aligned}
&&
\min_{\bm{\theta}}
&&&
\|\bm{X}\bm{\theta}-\bm{Y}\|.
\end{aligned}
\end{equation*}

To solve problem \eqref{P-mintheta1thetaN}, we first show that the problem can be reduced to a problem of the same form, but with one unknown parameter fewer.

\section{Elimination lemma}
\label{S-EL}

The next result offers a reduction approach to the problem, which provides the basis for the proposed solution.
\begin{lemma}
\label{L-mintheta1thetaN}
Solving problem \eqref{P-mintheta1thetaN} is equivalent to solving the problem
\begin{equation}
\begin{aligned}
&&
\min_{\theta_{1},\ldots,\theta_{N-1}}
&&
\max_{\substack{1\leq i<k\leq M\\|X_{iN}|+|X_{kN}|\ne0}}
\left|
\sum_{j=1}^{N-1}
\frac{X_{ij}X_{kN}-X_{kj}X_{iN}}{|X_{iN}|+|X_{kN}|}\theta_{j}
-
\frac{Y_{i}X_{kN}-Y_{k}X_{iN}}{|X_{iN}|+|X_{kN}|}
\right|
\end{aligned}
\label{P-mintheta1thetaN1}
\end{equation}
together with the inequality
\begin{multline}
\max_{\substack{1\leq i\leq M\\X_{iN}\ne0}}
\left(
-
\frac{\mu}{|X_{iN}|}
-
\sum_{j=1}^{N-1}
\frac{X_{ij}}{X_{iN}}\theta_{j}+\frac{Y_{i}}{X_{iN}}
\right)
\leq
\theta_{N}
\\\leq
\min_{\substack{1\leq i\leq M\\X_{iN}\ne0}}
\left(
\frac{\mu}{|X_{iN}|}
-
\sum_{j=1}^{N-1}
\frac{X_{ij}}{X_{iN}}\theta_{j}+\frac{Y_{i}}{X_{iN}}
\right),
\label{I-thetaN}
\end{multline}
where $\mu$ is the minimum of the objective function in problem \eqref{P-mintheta1thetaN1}, and the empty sums are defined to be zero.
\end{lemma}

\begin{proof}
To examine problem \eqref{P-mintheta1thetaN}, we first introduce an auxiliary unknown parameter $\lambda$ to rewrite the problem as follows:
\begin{equation*}
\begin{aligned}
&&
\min_{\theta_{1},\ldots,\theta_{N}}
&&&
\lambda,
\\
&&
\text{s.t.}
&&&
\max_{1\leq i\leq M}
\left|
\sum_{j=1}^{N}
X_{ij}\theta_{j}-Y_{i}
\right|
\leq
\lambda.
\end{aligned}
\end{equation*}

Note that the inequality constraint is readily represented as the system of inequalities
\begin{align*}
\lambda
&\geq
-
\sum_{j=1}^{N}
X_{ij}\theta_{j}+Y_{i},
\\
\lambda
&\geq
\mathbin{\phantom{-}}
\sum_{j=1}^{N}
X_{ij}\theta_{j}-Y_{i},
\qquad
i=1,\ldots,M;
\end{align*}
which, in particular, puts the problem into the form of a linear program.

Next, we continue rearrangement by solving for $\theta_{N}$ those inequalities in which $X_{iN}\ne0$ to write
\begin{align*}
-\theta_{N}
&\geq
\mathbin{\phantom{-}}
\frac{\lambda}{X_{iN}}
+
\sum_{j=1}^{N-1}
\frac{X_{ij}}{X_{iN}}\theta_{j}-\frac{Y_{i}}{X_{iN}},
\\
\theta_{N}
&\geq
\mathbin{\phantom{-}}
\frac{\lambda}{X_{iN}}
-
\sum_{j=1}^{N-1}
\frac{X_{ij}}{X_{iN}}\theta_{j}+\frac{Y_{i}}{X_{iN}},
\qquad
X_{iN}<0;
\\
\theta_{N}
&\geq
-
\frac{\lambda}{X_{iN}}
-
\sum_{j=1}^{N-1}
\frac{X_{ij}}{X_{iN}}\theta_{j}+\frac{Y_{i}}{X_{iN}},
\\
-\theta_{N}
&\geq
-
\frac{\lambda}{X_{iN}}
+
\sum_{j=1}^{N-1}
\frac{X_{ij}}{X_{iN}}\theta_{j}-\frac{Y_{i}}{X_{iN}},
\qquad
X_{iN}>0;
\qquad
i=1,\ldots,M.
\end{align*}

Coupling the inequalities with common left-hand sides and adding the inequalities for $X_{iN}=0$ yield
\begin{align*}
-\theta_{N}
&\geq
-\frac{\lambda}{|X_{iN}|}
+
\sum_{j=1}^{N-1}
\frac{X_{ij}}{X_{iN}}\theta_{j}-\frac{Y_{i}}{X_{iN}},
\\
\theta_{N}
&\geq
-\frac{\lambda}{|X_{iN}|}
-
\sum_{j=1}^{N-1}
\frac{X_{ij}}{X_{iN}}\theta_{j}+\frac{Y_{i}}{X_{iN}},
\qquad
X_{iN}\ne0;
\\
\lambda
&\geq
\left|
\sum_{j=1}^{N-1}
X_{ij}\theta_{j}-Y_{i}
\right|,
\qquad
X_{iN}=0;
\qquad
i=1,\ldots,M.
\end{align*}

By combining these inequalities for all $i=1,\ldots,M$, we obtain
\begin{equation}
\begin{aligned}
-\theta_{N}
&\geq
\max_{\substack{1\leq i\leq M\\X_{iN}\ne0}}
\left(
-
\frac{\lambda}{|X_{iN}|}
+
\sum_{j=1}^{N-1}
\frac{X_{ij}}{X_{iN}}\theta_{j}-\frac{Y_{i}}{X_{iN}}
\right),
\\
\theta_{N}
&\geq
\max_{\substack{1\leq i\leq M\\X_{iN}\ne0}}
\left(
-
\frac{\lambda}{|X_{iN}|}
-
\sum_{j=1}^{N-1}
\frac{X_{ij}}{X_{iN}}\theta_{j}+\frac{Y_{i}}{X_{iN}}
\right),
\\
\lambda
&\geq
\max_{\substack{1\leq i\leq M\\X_{iN}=0}}
\left|
\sum_{j=1}^{N-1}
X_{ij}\theta_{j}-Y_{i}
\right|.
\end{aligned}
\label{E-thetaN-lambda}
\end{equation}

The first two inequalities at \eqref{E-thetaN-lambda} result in the double inequality
\begin{multline*}
\max_{\substack{1\leq i\leq M\\X_{iN}\ne0}}
\left(
-
\frac{\lambda}{|X_{iN}|}
-
\sum_{j=1}^{N-1}
\frac{X_{ij}}{X_{iN}}\theta_{j}+\frac{Y_{i}}{X_{iN}}
\right)
\leq
\theta_{N}
\\\leq
-
\max_{\substack{1\leq i\leq M\\X_{iN}\ne0}}
\left(
-
\frac{\lambda}{|X_{iN}|}
+
\sum_{j=1}^{N-1}
\frac{X_{ij}}{X_{iN}}\theta_{j}-\frac{Y_{i}}{X_{iN}}
\right).
\end{multline*}

After replacing $\lambda$ by $\mu$ that denotes the minimum of the objective function and using the equality $\max(a,b)=-\min(-a,-b)$ to change from $\max$ to $\min$ in the right-hand side, the double inequality takes the form of \eqref{I-thetaN}.

The above double inequality defines a nonempty set of values for the unknown $\theta_{N}$ if, and only if, the following condition holds:
\begin{multline*}
\max_{\substack{1\leq i\leq M\\X_{iN}\ne0}}
\left(
-
\frac{\lambda}{|X_{iN}|}
-
\sum_{j=1}^{N-1}
\frac{X_{ij}}{X_{iN}}\theta_{j}+\frac{Y_{i}}{X_{iN}}
\right)
\\+
\max_{\substack{1\leq i\leq M\\X_{iN}\ne0}}
\left(
-
\frac{\lambda}{|X_{iN}|}
+
\sum_{j=1}^{N-1}
\frac{X_{ij}}{X_{iN}}\theta_{j}-\frac{Y_{i}}{X_{iN}}
\right)
\leq
0,
\end{multline*}
which is readily rearranged in the form of the inequality
\begin{multline*}
\max_{\substack{1\leq i,k\leq M\\X_{iN},X_{kN}\ne0}}
\left(
-
\frac{|X_{iN}|+|X_{kN}|}{|X_{iN}||X_{kN}|}\lambda
\right.
\\-
\sum_{j=1}^{N-1}
\frac{X_{ij}X_{kN}-X_{kj}X_{iN}}{X_{iN}X_{kN}}\theta_{j}
+
\left.
\frac{Y_{i}X_{kN}-Y_{k}X_{iN}}{X_{iN}X_{kN}}
\right)
\leq
0.
\end{multline*}

This inequality is equivalent to the system of inequalities
\begin{multline*}
-
\frac{|X_{iN}|+|X_{kN}|}{|X_{iN}||X_{kN}|}\lambda
-
\sum_{j=1}^{N-1}
\frac{X_{ij}X_{kN}-X_{kj}X_{iN}}{X_{iN}X_{kN}}\theta_{j}
+
\frac{Y_{i}X_{kN}-Y_{k}X_{iN}}{X_{iN}X_{kN}}
\leq
0,
\\
X_{iN},X_{kN}\ne0;
\qquad
1\leq i,k\leq M.
\end{multline*}

By solving these inequalities for $\lambda$, we obtain the system
\begin{multline*}
\lambda
\geq
-
\frac{|X_{iN}||X_{kN}|}{X_{iN}X_{kN}}
\left(
\sum_{j=1}^{N-1}
\frac{X_{ij}X_{kN}-X_{kj}X_{iN}}{|X_{iN}|+|X_{kN}|}\theta_{j}
-
\frac{Y_{i}X_{kN}-Y_{k}X_{iN}}{|X_{iN}|+|X_{kN}|}
\right),
\\
X_{iN},X_{kN}\ne0;
\qquad
1\leq i,k\leq M.
\end{multline*}

We now note that interchanging the indices $i$ and $k$ in the differences
\begin{equation*}
X_{ij}X_{kN}-X_{kj}X_{iN},
\qquad
Y_{i}X_{kN}-Y_{k}X_{iN}
\end{equation*}
changes the sign of these differences, and hence, the sign of the entire right-hand side of each inequality in the system. As a result, for every pair of indices $i$ and $k$, the system includes both the inequality
\begin{equation*}
\lambda
\geq
-
\frac{|X_{iN}||X_{kN}|}{X_{iN}X_{kN}}
\left(
\sum_{j=1}^{N-1}
\frac{X_{ij}X_{kN}-X_{kj}X_{iN}}{|X_{iN}|+|X_{kN}|}\theta_{j}
-
\frac{Y_{i}X_{kN}-Y_{k}X_{iN}}{|X_{iN}|+|X_{kN}|}
\right),
\end{equation*}
and the inequality 
\begin{multline*}
\lambda
\geq
-
\frac{|X_{kN}||X_{iN}|}{X_{kN}X_{iN}}
\left(
\sum_{j=1}^{N-1}
\frac{X_{kj}X_{iN}-X_{ij}X_{kN}}{|X_{kN}|+|X_{iN}|}\theta_{j}
-
\frac{Y_{k}X_{iN}-Y_{i}X_{kN}}{|X_{kN}|+|X_{iN}|}
\right)
\\=
\frac{|X_{iN}||X_{kN}|}{X_{iN}X_{kN}}
\left(
\sum_{j=1}^{N-1}
\frac{X_{ij}X_{kN}-X_{kj}X_{iN}}{|X_{iN}|+|X_{kN}|}\theta_{j}
-
\frac{Y_{i}X_{kN}-Y_{k}X_{iN}}{|X_{iN}|+|X_{kN}|}
\right).
\end{multline*}

After coupling the paired inequalities and considering that the equality $|X_{iN}X_{kN}|=|X_{iN}||X_{kN}|$ is valid, we rearrange the system as follows:
\begin{multline*}
\lambda
\geq
\left|
\sum_{j=1}^{N-1}
\frac{X_{ij}X_{kN}-X_{kj}X_{iN}}{|X_{iN}|+|X_{kN}|}\theta_{j}
-
\frac{Y_{i}X_{kN}-Y_{k}X_{iN}}{|X_{iN}|+|X_{kN}|}
\right|,
\\
X_{iN},X_{kN}\ne0;
\qquad
1\leq i,k\leq M.
\end{multline*}

Furthermore, we combine the inequalities for all $1\leq i,k\leq M$ and add the last inequality at \eqref{E-thetaN-lambda} to replace the condition $X_{iN},X_{kN}\ne0$ by that in the form $|X_{iN}|+|X_{kN}|\ne0$ and rewrite the system as one inequality
\begin{equation*}
\lambda
\geq
\max_{\substack{1\leq i,k\leq M\\|X_{iN}|+|X_{kN}|\ne0}}
\left|
\sum_{j=1}^{N-1}
\frac{X_{ij}X_{kN}-X_{kj}X_{iN}}{|X_{iN}|+|X_{kN}|}\theta_{j}
-
\frac{Y_{i}X_{kN}-Y_{k}X_{iN}}{|X_{iN}|+|X_{kN}|}
\right|.
\end{equation*}

We now observe that the term under the max operator is invariant under permutation of the indices $i$ and $k$, and is equal to zero if $i=k$. Therefore, we can reduce the set of indices defined as $1\leq i,k\leq M$ by that given by the condition $1\leq i<k\leq M$ and represent the lower bound on $\lambda$ as
\begin{equation*}
\lambda
\geq
\max_{\substack{1\leq i<k\leq M\\|X_{iN}|+|X_{kN}|\ne0}}
\left|
\sum_{j=1}^{N-1}
\frac{X_{ij}X_{kN}-X_{kj}X_{iN}}{|X_{iN}|+|X_{kN}|}\theta_{j}
-
\frac{Y_{i}X_{kN}-Y_{k}X_{iN}}{|X_{iN}|+|X_{kN}|}
\right|.
\end{equation*}

Since the minimum of $\lambda$ is bounded from below by the expression on the right-hand side, we need to find the minimum of this expression with respect to $\theta_{1},\ldots,\theta_{N-1}$, which leads to solving problem~\eqref{P-mintheta1thetaN1}.
\end{proof}

To conclude this section, we note that the reduced problem at \eqref{P-mintheta1thetaN1} has the same general form as problem \eqref{P-mintheta1thetaN} with the parameter $\theta_{N}$ eliminated. This offers a potential for solving the problem under study by recurrent implementation of lemma~\ref{L-mintheta1thetaN}. We discuss application of the lemma to derive direct solutions of problems of low dimension and to develop a recursive procedure to solve problems of arbitrary dimension in what follows.

\section{Solution of One- and Two-Parameter Regression Problems}
\label{S-SOTPRP}

We now apply the obtained result to derive direct, exact solutions to regression problems with one and two parameters. These solutions can be directly extended to problems with more parameters, which leads, however, to more complicated and bulky expressions, not presented here to save space. 

We start with one-parameter simple linear regression problems, which have well-known solutions, and then find a complete solution for a two-parameter linear regression problem. 

\subsection{One-Parameter Linear Regression Problems}

Let us suppose that, for given explanatory (independent) variables $X_{i}\in\mathbb{R}$ and response (dependent) variables $Y_{i}\in\mathbb{R}$, where $i=1,\ldots,M$, we find the unknown regression parameter $\theta\in\mathbb{R}$ that achieves the minimum
\begin{equation}
\begin{aligned}
&&
\min_{\theta}
&&&
\max_{1\leq i\leq M}|X_{i}\theta-Y_{i}|.
\end{aligned}
\label{P-minmaxXitheta1Yi}
\end{equation}

To solve the problem, we directly apply lemma~\ref{L-mintheta1thetaN} with $N=1$. Elimination of the empty sums in \eqref{P-mintheta1thetaN1} and \eqref{I-thetaN}, and substitution $X_{i1}=X_{i}$ for all $i=1,\ldots,M$ and $\theta_{1}=\theta$ yield the next results.
\begin{proposition}
\label{O-minmaxXitheta1Yi}
The minimum error in problem \eqref{P-minmaxXitheta1Yi} is equal to
\begin{equation*}
\mu
=
\max_{\substack{1\leq i<k\leq M\\|X_{i}|+|X_{k}|\ne0}}
\frac{|Y_{i}X_{k}-Y_{k}X_{i}|}{|X_{i}|+|X_{k}|},
\end{equation*}
and all solutions of the problem are given by the condition 
\begin{equation*}
\max_{\substack{1\leq i\leq M\\X_{i}\ne0}}
\left(
-
\frac{\mu}{|X_{i}|}
+
\frac{Y_{i}}{X_{i}}
\right)
\leq
\theta_{}
\leq
\min_{\substack{1\leq i\leq M\\X_{i}\ne0}}
\left(
\frac{\mu}{|X_{i}|}
+
\frac{Y_{i}}{X_{i}}
\right).
\end{equation*}
\end{proposition}

We now consider a special case of problem \eqref{P-minmaxXitheta1Yi} in the form
\begin{equation}
\begin{aligned}
&&
\min_{\theta}
&&&
\max_{1\leq i\leq M}|\theta-Y_{i}|.
\end{aligned}
\label{P-minmaxtheta1Yi}
\end{equation}

To handle the problem, we set $X_{i}=1$ for all $i=1,\ldots,M$ in the expressions obtained in proposition~\ref{O-minmaxXitheta1Yi}. Since $|X_{i}|+|X_{k}|=2\ne0$, the minimum error in problem \eqref{P-minmaxtheta1Yi} becomes
\begin{equation*}
\mu
=
\max_{1\leq i<k\leq M}
|Y_{i}-Y_{k}|/2
=
\max_{1\leq i,k\leq M}
|Y_{i}-Y_{k}|/2
=
\max_{1\leq i\leq M}Y_{i}/2
-
\min_{1\leq i\leq M}Y_{i}/2.
\end{equation*}

The solution $\theta$ is given by the condition
\begin{equation*}
-
\mu
+
\max_{1\leq i\leq M}
Y_{i}
\leq
\theta
\leq
\mu
+
\min_{1\leq i\leq M}
Y_{i},
\end{equation*}
which, after substitution of the above expression for $\mu$, leads to the unique result
\begin{equation*}
\theta
=
\max_{1\leq i\leq M}Y_{i}/2
+
\min_{1\leq i\leq M}Y_{i}/2.
\end{equation*}

\subsection{Two-Parameter Linear Regression Problem}

We now turn to two-parameter problems, which can be solved by twofold application of lemma~\ref{L-mintheta1thetaN}. To avoid cumbersome calculations, we concentrate on a special case in which, given variables $X_{i},Y_{i}\in\mathbb{R}$ for all $i=1,\ldots,M$, our aim is to find the parameters $\theta_{1},\theta_{2}\in\mathbb{R}$ to achieve
\begin{equation}
\begin{aligned}
&&
\min_{\theta_{1},\theta_{2}}
&&&
\max_{1\leq i\leq M}|\theta_{1}+X_{i}\theta_{2}-Y_{i}|.
\end{aligned}
\label{P-minmaxtheta1Xi2theta2Yi}
\end{equation}

\begin{proposition}
The minimum error in problem \eqref{P-minmaxtheta1Xi2theta2Yi} is equal to
\begin{equation}
\mu
=
\max_{\substack{1\leq i<k\leq M\\|X_{i}|+|X_{k}|\ne0}}
\max_{\substack{1\leq p<r\leq M\\X_{r}\ne X_{p}}}
\frac{|(Y_{i}X_{k}-Y_{k}X_{i})(X_{r}-X_{p})-(Y_{p}X_{r}-Y_{r}X_{p})(X_{k}-X_{i})|}
{(|X_{i}|+|X_{k}|)|X_{r}-X_{p}|+(|X_{p}|+|X_{r}|)|X_{k}-X_{i}|},
\label{E-mu}
\end{equation}
and all solutions of the problem are given by the conditions 
\begin{multline}
\max_{\substack{1\leq i<k\leq M\\X_{k}\ne X_{i}}}
\left(
-
\frac{|X_{i}|+|X_{k}|}{|X_{k}-X_{i}|}
\mu
+
\frac{Y_{i}X_{k}-Y_{k}X_{i}}{X_{k}-X_{i}}
\right)
\leq
\theta_{1}
\\\leq
\min_{\substack{1\leq i<k\leq M\\X_{k}\ne X_{i}}}
\left(
\frac{|X_{i}|+|X_{k}|}{|X_{k}-X_{i}|}
\mu
+
\frac{Y_{i}X_{k}-Y_{k}X_{i}}{X_{k}-X_{i}}
\right),
\label{I-theta1}
\end{multline}

\begin{equation}
\max_{\substack{1\leq i\leq M\\X_{i}\ne0}}
\left(
-
\frac{\mu}{|X_{i}|}
-
\frac{\theta_{1}}{X_{i}}+\frac{Y_{i}}{X_{i}}
\right)
\leq
\theta_{2}
\leq
\min_{\substack{1\leq i\leq M\\X_{i}\ne0}}
\left(
\frac{\mu}{|X_{i}|}
-
\frac{\theta_{1}}{X_{i}}+\frac{Y_{i}}{X_{i}}
\right).
\label{I-theta2}
\end{equation}
\end{proposition}

\begin{proof}
We apply lemma~\ref{L-mintheta1thetaN} with $N=2$, where we take $X_{i1}=1$ and $X_{i2}=X_{i}$ for all $i=1,\ldots,M$. As a result, problem \eqref{P-minmaxtheta1Xi2theta2Yi} reduces to the one-parameter problem
\begin{equation*}
\begin{aligned}
&&
\min_{\theta_{1}}
&&&
\max_{\substack{1\leq i<k\leq M\\|X_{i}|+|X_{k}|\ne0}}
\left|
\frac{X_{k}-X_{i}}{|X_{i}|+|X_{k}|}\theta_{1}
-
\frac{Y_{i}X_{k}-Y_{k}X_{i}}{|X_{i}|+|X_{k}|}
\right|
\end{aligned}
\end{equation*}
and box constraint for $\theta_{2}$ in the form of the double inequality at \eqref{I-theta2}, where $\mu$ is the minimum in the one-parameter problem.

We note that the objective function in the problem does not change if we replace the condition $1\leq i<k\leq M$ for indices over which the maximum is taken, by the extended condition $1\leq i,k\leq M$.

In a similar way as in lemma~\ref{L-mintheta1thetaN}, we first represent the one-parameter problem under consideration~as
\begin{equation*}
\begin{aligned}
&&
\min_{\theta_{1}}
&&&
\lambda,
\\
&&
\text{s.t.}
&&&
\max_{\substack{1\leq i,k\leq M\\|X_{i}|+|X_{k}|\ne0}}
\left|
\frac{X_{k}-X_{i}}{|X_{i}|+|X_{k}|}\theta_{1}
-
\frac{Y_{i}X_{k}-Y_{k}X_{i}}{|X_{i}|+|X_{k}|}
\right|
\leq
\lambda.
\end{aligned}
\end{equation*}

The inequality constraint in this problem is equivalent to the system of inequalities
\begin{align*}
\lambda
&\geq
-
\frac{X_{k}-X_{i}}{|X_{i}|+|X_{k}|}\theta_{1}
+
\frac{Y_{i}X_{k}-Y_{k}X_{i}}{|X_{i}|+|X_{k}|},
\\
\lambda
&\geq
\mathbin{\phantom{-}}
\frac{X_{k}-X_{i}}{|X_{i}|+|X_{k}|}\theta_{1}
-
\frac{Y_{i}X_{k}-Y_{k}X_{i}}{|X_{i}|+|X_{k}|},
\qquad
|X_{i}|+|X_{k}|\ne0;
\qquad
1\leq i,k\leq M.
\end{align*}

After solving the inequalities for $\theta_{1}$, we rewrite the system as
\begin{align*}
-\theta_{1}
&\geq
-
\frac{|X_{i}|+|X_{k}|}{|X_{k}-X_{i}|}
\lambda
-
\frac{Y_{i}X_{k}-Y_{k}X_{i}}{X_{k}-X_{i}},
\\
\theta_{1}
&\geq
-
\frac{|X_{i}|+|X_{k}|}{|X_{k}-X_{i}|}
\lambda
+
\frac{Y_{i}X_{k}-Y_{k}X_{i}}{X_{i1}X_{k}-X_{k1}X_{i}},
\qquad
X_{k}\ne X_{i};
\\
\lambda
&\geq
\frac{|Y_{i}X_{k}-Y_{k}X_{i}|}{|X_{i}|+|X_{k}|},
\qquad
X_{k}=X_{i};
\qquad
|X_{i}|+|X_{k}|\ne0;
\qquad
1\leq i,k\leq M.
\end{align*}

By combining these inequalities, we obtain
\begin{equation}
\begin{aligned}
-\theta_{1}
&\geq
\max_{\substack{1\leq i,k\leq M\\|X_{i}|+|X_{k}|\ne0\\X_{k}\ne X_{i}}}
\left(
-
\frac{|X_{i}|+|X_{k}|}{|X_{k}-X_{i}|}
\lambda
-
\frac{Y_{i}X_{k}-Y_{k}X_{i}}{X_{k}-X_{i}}
\right),
\\
\theta_{1}
&\geq
\max_{\substack{1\leq i,k\leq M\\|X_{i}|+|X_{k}|\ne0\\X_{k}\ne X_{i}}}
\left(
-
\frac{|X_{i}|+|X_{k}|}{|X_{k}-X_{i}|}
\lambda
+
\frac{Y_{i}X_{k}-Y_{k}X_{i}}{X_{k}-X_{i}}
\right),
\\
\lambda
&\geq
\max_{\substack{1\leq i,k\leq M\\|X_{i}|+|X_{k}|\ne0\\X_{k}=X_{i}}}
\frac{|Y_{i}X_{k}-Y_{k}X_{i}|}{|X_{i}|+|X_{k}|}
=
\max_{\substack{1\leq i,k\leq M\\|X_{i}|\ne0}}
\frac{|Y_{i}-Y_{k}|}{2}.
\end{aligned}
\label{E-theta1-lambda}
\end{equation}

The first two inequalities yield the double inequality
\begin{multline*}
\max_{\substack{1\leq i,k\leq M\\|X_{i}|+|X_{k}|\ne0\\X_{k}\ne X_{i}}}
\left(
-
\frac{|X_{i}|+|X_{k}|}{|X_{k}-X_{i}|}
\lambda
+
\frac{Y_{i}X_{k}-Y_{k}X_{i}}{X_{k}-X_{i}}
\right)
\leq
\theta_{1}
\\\leq
\min_{\substack{1\leq i,k\leq M\\|X_{i}|+|X_{k}|\ne0\\X_{k}\ne X_{i}}}
\left(
\frac{|X_{i}|+|X_{k}|}{|X_{k}-X_{i}|}
\lambda
+
\frac{Y_{i}X_{k}-Y_{k}X_{i}}{X_{k}-X_{i}}
\right).
\end{multline*}

Since, under the condition $X_{k}\ne X_{i}$, the condition $|X_{i}|+|X_{k}|\ne0$ holds as well, we exclude the latter one. Observing that the terms under the $\max$ and $\min$ operators are invariant under permutation of $i$ and $k$, we adjust the condition on indices to write the box constraint for $\theta_{1}$ as \eqref{I-theta1}.

The feasibility condition for the box constraint to be valid for $\theta_{1}$ takes the form of the inequality
\begin{multline*}
\max_{\substack{1\leq i,k\leq M\\X_{k}\ne X_{i}}}
\max_{\substack{1\leq p,r\leq M\\X_{r}\ne X_{p}}}
\left(
-
\left(
\frac{|X_{i}|+|X_{k}|}{|X_{k}-X_{i}|}
+
\frac{|X_{p}|+|X_{r}|}{|X_{r}-X_{p}|}
\right)
\lambda
\right.
\\+
\left.
\frac{Y_{i}X_{k}-Y_{k}X_{i}}{X_{k}-X_{i}}
-
\frac{Y_{p}X_{r}-Y_{r}X_{p}}{X_{r}-X_{p}}
\right)
\leq0.
\end{multline*}

As before, we represent this inequality as the system of inequalities for each $i,k$ and $p,r$, and then solve these inequalities for $\lambda$. After combining the solutions back into one inequality and adding the last inequality at \eqref{E-theta1-lambda}, we obtain
\begin{equation*}
\lambda
\geq
\max_{\substack{1\leq i,k\leq M\\|X_{i}|+|X_{k}|\ne0}}
\max_{\substack{1\leq p,r\leq M\\X_{r}\ne X_{p}}}
\frac{|(Y_{i}X_{k}-Y_{k}X_{i})(X_{r}-X_{p})-(Y_{p}X_{r}-Y_{r}X_{p})(X_{k}-X_{i})|}
{(|X_{i}|+|X_{k}|)|X_{r}-X_{p}|+(|X_{p}|+|X_{r}|)|X_{k}-X_{i}|},
\end{equation*}
where the expression on the right-hand side determines the minimum $\mu$.

Finally, we note that the fractional term under maximization is invariant with respect to interchanging the indices $i$ and $k$, as well as $p$ and $r$. Therefore, we can replace the conditions $1\leq i,k\leq M$ and $1\leq p,r\leq M$ by the conditions $1\leq i<k\leq M$ and $1\leq p<r\leq M$, which yields the representation for $\mu$ in the form of \eqref{E-mu}.
\end{proof}

\section{General Solution Procedure}
\label{S-GSP}

We now use lemma~\ref{L-mintheta1thetaN} to derive a complete solution of problem \eqref{P-mintheta1thetaN} by performing a series of reduction steps, each eliminating an unknown parameter in the problem and determining a box constraint for this parameter. We observe that the elimination of a parameter from the objective function as described in lemma~\ref{L-mintheta1thetaN} leaves the general form of the problem unchanged. Therefore, we can repeat the elimination over and over again until the function has no more parameters, and thus becomes a constant that shows the minimum of the objective function in the initial problem.

At the same time, together with the elimination of a parameter from the objective function, lemma~\ref{L-mintheta1thetaN} offers a box constraint for this parameter, represented in terms of those parameters which are retained in the function. We see that the last constraint does not depend on any other parameters, and thus is directly given by a double inequality with constant values on both sides. As a result, we can take the box constraints in the order from the last constraint to the first, which yields a system of double inequalities that completely determines the solution set of the problem.  

We are in a position to describe the solution procedure formally in more detail. The procedure follows a direct solution method that examines the unknown parameters in reversal order, starting from the parameter $\theta_{N}$ and going backward to the parameter $\theta_{1}$. Let $n$ be the number of parameters in the objective function in the current step of the procedure.

Initially, we take $n=N$ and set $M_{n}=M$. For all $i=1,\ldots,M_{n}$ and $j=1,\ldots,n$, we define
\begin{equation*}
X_{ij}^{n}
=
X_{ij},
\qquad
Y_{i}^{n}
=
Y_{i}.
\end{equation*} 

We also introduce the matrix-vector notation
\begin{equation*}
\bm{X}_{n}
=
\left(
\begin{array}{ccc}
X_{11}^{n} & \dots & X_{1n}^{n}
\\
\vdots & \ddots & \vdots
\\
X_{M_{n},1}^{n} & \dots & X_{M_{n},n}^{n}
\end{array}
\right),
\qquad
\bm{Y}_{n}
=
\left(
\begin{array}{c}
Y_{1}^{n}
\\
\vdots
\\
Y_{M_{n}}^{n}
\end{array}
\right),
\qquad
\bm{\theta}_{n}
=
\left(
\begin{array}{c}
\theta_{1}
\\
\vdots
\\
\theta_{n}
\end{array}
\right).
\end{equation*}

For each $n=N,N-1,\ldots,1$, the procedure produces a two-fold outcome: the reduction of the current problem by eliminating an unknown parameter, and the derivation of a box constraint for the eliminated parameter.

\subsection{Elimination of Parameters}

Assuming that the norm sign in what follows stands for the Chebyshev norm, we start with eliminating the parameter $\theta_{n}$ from the problem
\begin{equation*}
\begin{aligned}
&&
\min_{\theta_{1},\ldots,\theta_{n}}
&&
\max_{1\leq i\leq M_{n}}
\left|
\sum_{j=1}^{n}
X_{ij}^{n}\theta_{j}-Y_{i}^{n}
\right|
\;
=
\;
\min_{\bm{\theta}_{n}}
&&
\|\bm{X}_{n}\bm{\theta}_{n}-\bm{Y}_{n}\|.
\end{aligned}
\end{equation*}

It follows from lemma~\ref{L-mintheta1thetaN} that, as a result of this elimination, the problem reduces, if $n>1$, to the~problem
\begin{equation*}
\begin{aligned}
&&
\min_{\theta_{1},\ldots,\theta_{n-1}}
&&
\max_{1\leq i\leq M_{n-1}}
\left|
\sum_{j=1}^{n-1}
X_{ij}^{n-1}\theta_{j}-Y_{i}^{n-1}
\right|
\;
=
\;
\min_{\bm{\theta}_{n-1}}
&&
\|\bm{X}_{n-1}\bm{\theta}_{n-1}-\bm{Y}_{n-1}\|,
\end{aligned}
\end{equation*}
or degenerates, if $n=1$, to the constant
\begin{equation*}
\max_{1\leq i\leq M_{0}}
\left|
Y_{i}^{0}
\right|
\;
=
\;
\|\bm{Y}_{0}\|.
\end{equation*}

We now exploit the representation of problem \eqref{P-mintheta1thetaN1} to establish formulas of recalculating the objective function when changing to the reduced problem.

First, we consider the condition for indices in \eqref{P-mintheta1thetaN1}, which takes the form $1\leq i<k\leq M_{n}$. We assume that the pairs of indices $(i,k)$ defined by the condition are listed in the order of the sequence
\begin{equation*}
(1,2),\ldots,(1,M_{n}),(2,3),\ldots,(2,M_{n}),\ldots,(M_{n}-1,M_{n}).
\end{equation*}

It is not difficult to verify by direct substitution that each fixed pair $(i,k)$  in this sequence has the number (index) calculated as
\begin{equation*}
M_{n}(i-1)-i(i-1)/2+k-1.
\end{equation*}

Furthermore, we use \eqref{P-mintheta1thetaN1} to define, for all $i$ and $k$ such that $1\leq i<k\leq M_{n}$ and for all $j=1,\ldots,n-1$, the recurrent formulas
\begin{equation}
\begin{aligned}
X_{M_{n}(i-1)-i(i-1)/2+k-1,j}^{n-1}
&=
\begin{cases}
\frac{X_{ij}^{n}X_{kn}^{n}-X_{kj}^{n}X_{in}^{n}}{|X_{in}^{n}|+|X_{kn}^{n}|},
&
\text{if $|X_{in}^{n}|+|X_{kn}^{n}|\ne0$};
\\
0
&
\text{otherwise};
\end{cases}
\\
Y_{M_{n}(i-1)-i(i-1)/2+k-1}^{n-1}
&=
\begin{cases}
\frac{Y_{i}^{n}X_{kn}^{n}-Y_{k}^{n}X_{in}^{n}}{|X_{in}^{n}|+|X_{kn}^{n}|},
&
\text{if $|X_{in}^{n}|+|X_{kn}^{n}|\ne0$};
\\
0
&
\text{otherwise}.
\end{cases}
\end{aligned}
\label{E-Xn-Yn}
\end{equation}

Note that if $|X_{in}^{n}|+|X_{kn}^{n}|=0$, then the above formulas produce a row of zeros that corresponds to a zero term, which does not contribute to the objective function. We assume that all such zero rows are removed, and the rest of the rows are renumbered (reindexed) to preserve continual enumeration. 

We denote the number of nonzero rows by $M_{n-1}$ and observe that
\begin{equation*}
M_{n-1}
\leq
M_{n}(M_{n}-1)/2.
\end{equation*}

Finally, we take the numbers $X_{ij}^{n-1}$ and $Y_{i}^{n-1}$ with $i=1,\ldots,M_{n-1}$ and $j=1,\ldots,n-1$ to form the matrix and vector
\begin{equation*}
\bm{X}_{n-1}
=
(X_{ij}^{n-1}),
\qquad
\bm{Y}_{n-1}
=
(Y_{i}^{n-1}),
\end{equation*}
which completely determine the objective function in the reduced problem. Specifically, the reduced problem for $n=1$ degenerates into the constant
\begin{equation}
\mu
=
\|\bm{Y}_{0}\|
\label{E-muY0},
\end{equation}
representing the minimum of the objective function in the initial problem.

\subsection{Derivation of Box Constraints}

We take the double inequality at \eqref{I-thetaN} and adjust it to write the box constraint for the parameter $\theta_{n}$ in the form
\begin{multline*}
\max_{\substack{1\leq i\leq M_{n}\\X_{in}^{n}\ne0}}
\left(
-
\frac{\mu}{|X_{in}^{n}|}
-
\sum_{j=1}^{n-1}
\frac{X_{ij}^{n}}{X_{in}^{n}}\theta_{j}
+
\frac{Y_{i}^{n}}{X_{in}^{n}}
\right)
\leq
\theta_{n}
\\\leq
\min_{\substack{1\leq i\leq M_{n}\\X_{in}^{n}\ne0}}
\left(
\frac{\mu}{|X_{in}^{n}|}
-
\sum_{j=1}^{n-1}
\frac{X_{ij}^{n}}{X_{in}^{n}}\theta_{j}
+
\frac{Y_{i}^{n}}{X_{in}^{n}}
\right),
\end{multline*}
where $\mu$ denotes the minimum value of the objective function in problem \eqref{P-mintheta1thetaN}.

To represent this inequality constraint in vector form, we introduce the following notation. First, for all $i=1,\ldots,M_{n}$ such that $X_{in}^{n}\ne0$ and all $j=1,\ldots,n-1$, we define
\begin{equation}
T_{ij}^{n}
=
\frac{X_{ij}^{n}}{X_{in}^{n}},
\qquad
L_{i}^{n}
=
\frac{Y_{i}^{n}}{X_{in}^{n}}
-
\frac{\mu}{|X_{in}^{n}|},
\qquad
U_{i}^{n}
=
\frac{Y_{i}^{n}}{X_{in}^{n}}
+
\frac{\mu}{|X_{in}^{n}|}.
\label{E-Tijn-Lin-Uin}
\end{equation}

We note that all indices $i$ with $X_{in}^{n}=0$ are not taken into account when calculating the maximum and minimum in the double inequality, and hence are excluded from the index set $1,\ldots,M_{n}$. Assuming that the rest of indices are renumbered to preserve continual enumeration, we introduce the matrix and column vectors
\begin{equation*}
\bm{T}_{n}
=
(T_{ij}^{n}),
\qquad
\bm{L}_{n}
=
(L_{i}^{n}),
\qquad
\bm{U}_{n}
=
(U_{i}^{n}).
\end{equation*}

With this matrix-vector notation, we write the box constraint, if $n>1$, as the double inequality
\begin{equation}
\max(\bm{L}_{n}-\bm{T}_{n}\bm{\theta}_{n-1})
\leq
\theta_{n}
\leq
\min(\bm{U}_{n}-\bm{T}_{n}\bm{\theta}_{n-1}),
\label{I-thetan_vector}
\end{equation}
and, if $n=1$, as the inequality
\begin{equation}
\max(\bm{L}_{1})
\leq
\theta_{1}
\leq
\min(\bm{U}_{1}),
\label{I-theta1_vector}
\end{equation}
where $\max$ and $\min$ symbols are thought of as operators that calculate the maximum and minimum over all elements of corresponding vectors.

\subsection{Solution Algorithm}

We summarize the above consideration in the form of a computational algorithm to solve problem \eqref{P-mintheta1thetaN} in a finite number of steps. The algorithm includes two sequential phases: backward elimination and forward determination (substitution) of the unknown parameters.

The backward elimination starts with $n=N$ by setting $M_{n}=M$ and
\begin{equation*}
\bm{X}_{n}
=
\bm{X},
\qquad
\bm{Y}_{n}
=
\bm{Y}.
\end{equation*} 

Furthermore, for each $n=N,\ldots,1$, the matrix $\bm{X}_{n}$ and vector $\bm{Y}_{n}$ are used as described above to obtain the values of $\bm{X}_{n-1}$ and $\bm{Y}_{n-1}$ if $n>1$ or the value of $\bm{Y}_{0}$ if $n=1$. As supplementary results, the matrices $\bm{T}_{n}$ are also evaluated from the matrices $\bm{X}_{n}$.

The backward elimination completes at $n=1$ by calculating the minimum value of the objective function, given by $\mu=\|\bm{Y}_{0}\|$.

The forward determination first uses the obtained minimum $\mu$, matrix $\bm{X}_{1}$ and vector $\bm{Y}_{1}$ to calculate the vectors $\bm{L}_{1}$ and $\bm{U}_{1}$ and then evaluate the box constraint for the unknown $\theta_{1}$ in the form of double inequality at \eqref{I-theta1_vector}. Then, for each $n=2,\ldots,N$, the vectors $\bm{L}_{n}$ and $\bm{U}_{n}$ are calculated from $\mu$, $\bm{X}_{n}$ and $\bm{Y}_{n}$ to represent the box constraints for $\theta_{n}$ as in \eqref{I-thetan_vector}.

Note that the bounds in the box constraint for the parameter $\theta_{1}$ are explicitly defined by constants, whereas the bounds for each parameter $\theta_{n}$ with $n>1$ are defined as linear functions of the previous parameters $\theta_{1},\ldots,\theta_{n-1}$. Therefore, we can first fix a value to satisfy the box constraint for $\theta_{1}$ and then substitute this value into the box constraint for $\theta_{2}$ to obtain explicit bounds given by constants. By repeating such calculations to fix a value for a parameter with explicit bounds and to evaluate bounds for the next parameter, we can successively determine a solution of the problem.

\subsection{Computational Complexity}

The most computationally intensive and memory demanding component of the algorithm, which determines the overall rate of computational complexity and memory requirement, is the calculation of entries in the matrices $\bm{X}_{n}$ and vectors $\bm{Y}_{n}$ for all $n=N-1,N-2,\ldots,1$, and vector $\bm{Y}_{0}$ by using \eqref{E-Xn-Yn}. Though calculating one entry involves a few simple operations, the number of all entries grows very fast as $M$ and $N$ increase.

To derive a rough estimate for the number of entries in all matrices, we first evaluate the number of rows in each matrix. Assuming that the matrix $\bm{X}_{N}=\bm{X}$ has $M$ rows, we see that the number of rows in $\bm{X}_{N-1}$ is bounded from above by
\begin{equation*}
M_{N-1}
\leq
M_{N}(M_{N}-1)/2
=
M(M-1)/2
<
M^{2}/2
=
2(M/2)^{2}.
\end{equation*}

Recursive application of this estimate yields an upper bound for the number of rows in the matrices $X_{N-l}$ for each $l=1,\ldots,N-1$ in the form 
\begin{equation*}
M_{N-l}
\leq
M_{N-l+1}(M_{N-l+1}-1)/2
<
M_{N-l+1}^{2}/2
<
2(M/2)^{{2}^{l}}.
\end{equation*}

At the last step with $l=N$, we calculate the vector $\bm{Y}_{0}$, in which the number of entries is no more~than
\begin{equation*}
M_{0}
\leq
M_{1}(M_{1}-1)/2
<
M_{1}^{2}/2
<
2(M/2)^{{2}^{N}}.
\end{equation*}

Since we have $n+1$ columns in the matrix $\bm{X}_{n}$ together with the vector $\bm{Y}_{n}$, the overall number of the entries in all steps can be estimated as
\begin{equation*}
\sum_{l=1}^{N}
(N-l+1)M_{N-l}
<
\sum_{l=1}^{N}
2(N-l+1)(M/2)^{{2}^{l}}.
\end{equation*}

We denote the upper bound on the right-hand side by $C(N,M)$ and observe that this bound increases polynomially with respect to $M$, and double exponentially with $N$. For problems with few parameters, the value of $C(N,M)$ seems to be rather acceptable. Specifically, in the three-parameter case with $N=3$ and $M=10,20,50$, we have $C(3,10)=783,900$, $C(3,20)=200,040,600$, and $C(3,50)=305,177,347,500$. A further increase of the number of parameters $N$ results in a rapid rise in the value of $C(N,M)$, as the following examples show: $C(4,10)=305,177,347,700$, $C(5,10)=610,353,911,500$.

Note that the actual number of entries to calculate is fewer than that given by the bound $C(N,M)$. As it follows from \eqref{E-Xn-Yn}, this number can be further reduced if the matrices $\bm{X}_{n}$, or at least the initial matrix $\bm{X}_{N}=\bm{X}$, have many zero entries (sparse matrices). It is clear that, if a matrix has zero entries in a column other than the last one, the columns (and related parameters) can be renumbered to put the column with zero entries on the last place where zero entries can reduce computations. As an example of the case with good chances of having sparse matrices $\bm{X}_{n}$, one can consider problems where the initial matrix $\bm{X}$ has entries that take values from the set $\{-1,0,1\}$.

Finally, we observe that the calculations by formula \eqref{E-Xn-Yn} can be performed for different entries quite independently, which offers strong potential for parallel implementation of the procedure on parallel and vector computers to provide more computational and memory resources, and hence to extend applicability to problems of higher dimensions.

\section{Software Implementation and Numerical Examples}
\label{S-SINE}

We conclude with comments on software implementation of the solution procedure and illustrative numerical examples of low dimensions that demonstrate the computational technique involved in the solution. For the sake of illustration, we concentrate on application to problems with exact input data given by integer (rational) numbers to find explicit rational solutions.

To obtain exact solutions, the procedure has been coded for serial symbolic computations in the MATLAB (Release R2020b) environment as a collection of functions that calculate all intermediate matrices and vectors, as well as provide the overall functionality of the algorithm. The numerical experiments were conducted on a custom computer with a 4-core 8-thread Intel Xeon E3-1231 v3 CPU at 3.40GHz and 32GB of DDR3 RAM, running Windows~10 Enterprise 64-bit OS.

\begin{example}
Let us take $N=3$ and $M=4$ and consider the approximation (regression) problem with
\begin{equation*}
\bm{X}
=
\left(
\begin{array}{rrr}
 3 & -1 & 2
\\
-1 & -2 & 2
\\
-2 & 3 & -1
\\
 0 & 2 & -1
\end{array}
\right),
\qquad
\bm{Y}
=
\left(
\begin{array}{r}
 2
\\
 1
\\
-1
\\
 0
\end{array}
\right).
\end{equation*}

We start with the backward elimination phase by setting $\bm{X}_{3}=\bm{X}$ and $\bm{Y}_{3}=\bm{Y}$. We use \eqref{E-Xn-Yn} to calculate the entries in
\begin{equation*}
\bm{X}_{2}
=
\left(
\begin{array}{rr}
2 & 1/2
\\
1/3 & -5/3
\\
-1 & -1
\\
5/3 & -4/3
\\
1/3 & -2/3
\\
1 & -1/2\\
\end{array}
\right),
\qquad
\bm{Y}_{2}
=
\left(
\begin{array}{r}
1/2
\\
0
\\
-2/3
\\
1/3
\\
-1/3
\\
1/2
\end{array}
\right),
\end{equation*}
and then the entries in
\begin{equation*}
\bm{X}_{1}
=
\left(
\begin{array}{r}
-21/13
\\
-1
\\		
-21/11
\\
-9/7
\\	
-3/2
\\	
-3/4
\\	
7/9
\\	
1/7
\\	
9/13
\\	
9/7
\\	
3/5
\\	
1
\\		
-1/3
\\	
3/11
\\	
3/7
\end{array}
\right),
\qquad
\bm{Y}_{1}
=
\left(
\begin{array}{r}
-5/13
\\
-1/9
\\
-5/11
\\
-1/7
\\
-1/2
\\
-5/12
\\
5/27
\\
-5/21
\\
5/13
\\
11/21
\\
1/15
\\
5/9
\\	
-1/3
\\
3/11
\\
3/7
\end{array}
\right).
\end{equation*}

Next, we obtain the vector $\bm{Y}_{0}$, which appears to have $105$ entries and thus is not shown here to save space. Evaluating the maximum entry of $\bm{Y}_{0}$ as the minimum of the objective function according to \eqref{E-muY0} yields
\begin{equation*}
\mu
=
2/7.
\end{equation*}

In parallel with evaluating the entries of $\bm{X}_{n}$ and $\bm{Y}_{n}$, we apply \eqref{E-Tijn-Lin-Uin} to find the entries in
\begin{equation*}
\bm{T}_{3}
=
\left(
\begin{array}{rc}
-3/2 & 1/2
\\
1/2 & 1
\\
-2 & 3
\\
0 & 2
\end{array}
\right),
\qquad
\bm{T}_{2}
=
\left(
\begin{array}{r}
-4
\\
1/5
\\
-1
\\
5/4
\\
1/2
\\
2
\end{array}
\right).
\end{equation*}

Substitution of the obtained minimum $\mu=2/7$ yields
\begin{gather*}
\bm{L}_{3}
=
\left(
\begin{array}{r}
6/7
\\
5/14
\\
5/7
\\
-2/7
\end{array}
\right),
\qquad
\bm{U}_{3}
=
\left(
\begin{array}{c}
8/7
\\
9/14
\\
9/7
\\
2/7
\end{array}
\right),
\\
\bm{L}_{2}
=
\left(
\begin{array}{r}
3/7
\\
-6/35
\\
8/21
\\
-13/28
\\
1/14
\\
-11/7
\end{array}
\right),
\qquad
\bm{U}_{2}
=
\left(
\begin{array}{r}
11/7
\\
6/35
\\
20/21
\\
-1/28
\\
13/14
\\
-3/7
\end{array}
\right).
\end{gather*}

Finally, we calculate the vectors
\begin{equation*}
\bm{L}_{1}
=
\left(
\begin{array}{r}
3/49
\\
-11/63
\\
13/147
\\
-1/9
\\
1/7
\\
11/63
\\
-19/147
\\
-11/3
\\
1/7
\\
5/27
\\
-23/63
\\
17/63
\\
1/7
\\
-1/21
\\
1/3
\end{array}
\right),
\qquad
\bm{U}_{1}
=
\left(
\begin{array}{c}
61/147
\\
25/63
\\
19/49
\\
1/3
\\
11/21
\\
59/63
\\
89/147
\\
1/3
\\
61/63
\\
17/27
\\
37/63
\\
53/63
\\
13/7
\\
43/21
\\
5/3
\end{array}
\right).
\end{equation*}

The forward determination (substitution) phase of the procedure involves application of \eqref{I-theta1_vector} and then \eqref{I-thetan_vector} to obtain the unique solution
\begin{equation*}
\theta_{1}
=
1/3,
\qquad
\theta_{2}
=
5/21,
\qquad
\theta_{3}
=
16/21.
\end{equation*}

The actual number of entries in the matrices $\bm{X}_{n}$ and vectors $\bm{Y}_{n}$ to calculate in the problem is $153$, which is fewer than the upper bound given by $C(3,4)=600$. The computer time to obtain the exact solution by symbolic computation averages $2.3$~sec.
\end{example}

\begin{example}
Consider the problem with $N=3$ and $M=10$, where the input data are given by
\begin{equation*}
\bm{X}
=
\left(
\begin{array}{rrr}
3 & -1 & 2
\\
-1 & -2 & 2
\\
-2 & 3 & -1
\\
0 & 2 & -1
\\
1 & 2 & -1
\\
3 & 1 & 0
\\
1 & 1 & -1
\\
-1 & -1 & 2
\\
0 & 3 & 1
\\
2 & 1 & 0
\end{array}
\right),
\qquad
\bm{Y}
=
\left(
\begin{array}{r}
2
\\
1
\\
-1
\\
0
\\
1
\\
-1
\\
2
\\
0
\\
1
\\
3
\end{array}
\right).
\end{equation*}

The application of the procedure yields the unique solution
\begin{equation*}
\theta_{1}
=
1/9,
\qquad
\theta_{2}
=
13/18,
\qquad
\theta_{3}
=
8/9.
\end{equation*}

The matrices $\bm{X}_{n}$ and vectors $\bm{Y}_{n}$ involved in calculations have $444,280$ entries (while the corresponding upper bound is $C(3,10)=783,900$). The symbolic computations take about $52$~min of computer time.
\end{example}

\section{Conclusions}
\label{S-C}

A direct computational technique has been proposed to solve discrete linear Chebyshev approximation problems, which find wide application in various areas, including the least maximum absolute deviation regression in statistics. First, we have shown that the problem under consideration can be reduced by eliminating an unknown parameter to a problem with less number of unknowns and a box constraint for the parameter eliminated. This result was used to obtain direct solutions to linear regression problems with one and two parameters.

To solve approximation problems of arbitrary dimension, we have developed a direct solution procedure that consists of two parts: backward elimination and forward substitution of the unknown parameters. The direct solution is of particular interest in the problems when analytical solutions are desired, whereas the use of iterative algorithms may be inappropriate or inadequate. We have estimated the computational complexity of the procedure, discussed its MATLAB implementation intended to provide exact solutions by symbolic computations, and presented numerical examples.

Possible lines of further investigation can include modification and improvement of the algorithm to reduce computational complexity in solving problems in both exact (rational) and inexact (floating point) form. The development of parallel implementations of the algorithm to speed up calculations is also of interest.

\bibliographystyle{abbrvurl}

\bibliography{Using_parameter_elimination_to_solve_discrete_linear_Chebyshev_approximation_problems}

\end{document}